\documentclass{amsart}
\usepackage{amsmath,amssymb,enumerate}
\usepackage[all]{xy}
\usepackage{setspace}
\usepackage[T1]{fontenc}
\usepackage{textcomp} 
\usepackage{stmaryrd}
\usepackage[margin=3cm]{geometry}

\newcommand{\tmop}[1]{\ensuremath{\operatorname{#1}}}

\newenvironment{enumerateroman}{\begin{enumerate}[(i)] }{\end{enumerate}}
\newenvironment{acknowledgements}{\noindent\emph{Acknowledgements.\ }}{}

\newtheorem{theorem}{Theorem}[section]
\newtheorem{cor}[theorem]{Corollary}
\newtheorem{lemma}[theorem]{Lemma}
\newtheorem{prop}[theorem]{Proposition}
\theoremstyle{definition}
\newtheorem{definition}[theorem]{Definition}

\newtheorem{remark}[theorem]{Remark}

\numberwithin{equation}{section}

\newcommand{\mc}[1]{\mathcal{#1}}

\newcommand{\ra}{\rightarrow}

\newcommand{\mat}[2][cccccccccccccccccccccccccccccccccccc]{\left(\begin{array}{#1}#2  \end{array}
  \right)}

\newcommand{\C}[1]{\mathbb{C}^{#1}}

\newcommand{\bs}[1]{\boldsymbol{#1}}

\newcommand{\ch}[1]{\tmop{Ch}(#1)}

\newcommand{\ucp}[2]{\tmop{UCP}(#1, #2)}

\newcommand{\os}[1]{\tmop{OS}(#1)}
\newcommand{\oa}[1]{\tmop{OS}(#1)}

\begin{document}
\begin{abstract}
In 2006, Arveson resolved a long-standing problem by showing that for
any element $x$ of a separable self-adjoint unital subspace $S\subseteq B(H)$,
$\|x\|=\sup\|\pi(x)\|$, where $\pi$ runs over the boundary
representations for $S$.  Here we show that ``sup'' can be replaced by
``max''. This implies that the Choquet boundary for a separable
operator system is a boundary in the classical sense; a similar result is obtained in terms of pure matrix states
when $S$ is not assumed to be separable.  For matrix convex sets associated to operator systems in matrix algebras, we apply the above results to improve the Webster-Winkler Krein-Milman theorem.
\end{abstract}
\title[Boundary representations and pure cp maps]{Boundary
  representations and pure completely positive maps}
\author{Craig Kleski}
\address{Department of Mathematics\\ University of Virginia\\
  Charlottesville, VA 22904}
\email{ckleski@virginia.edu} 
\date{\today}
\maketitle

\section{Introduction}
Let $B(H)$ be the bounded linear operators on a complex Hilbert space $H$ and let $S\subseteq B(H)$ be a concrete \emph{operator system}: a self-adjoint unital
linear subspace.  We denote by $C^{\ast}(S)$ the $C^{\ast}$-algebra
generated by $S$ in $B(H)$.   A unital completely positive (ucp) map on $S$ that extends uniquely
as a ucp map to a representation $\pi$ of $C^{\ast}(S)$ has the \emph{unique extension
  property} (UEP); if this representation is irreducible, we say that $\pi$
is a \emph{boundary representation} for $S$. In other words, an
irreducible representation $\pi$ of $C^{\ast}(S)$ is a boundary
representation for $S$ if the only ucp extension of $\pi|_{S}$ is
$\pi$.  Let $\partial_{S}$ denote the set of boundary
representations for $S$.
Arveson (\cite{Arveson08}) proved that if $S$ is separable, then $S$ has
sufficiently many boundary representations in the following sense:
for any $n$ and any $(s_{ij})\in M_n(S)$, 
\begin{align*}
\|(s_{ij})\| &= \sup_{\pi\in \partial_{S}}\|(\pi(s_{ij}))\|.
\end{align*}
We improve that result by showing in Theorem~\ref{mainresult} that we can replace the supremum
in the above with a maximum. 

A similar though not identical result can be obtained when $S$ is not assumed to be
separable.  Let $\tmop{CP}(S,B(K))$ denote
the cone of completely positive (cp) maps from $S$ to $B(K)$, and let
$\tmop{UCP}(S,B(K))$ be the convex subset of cp maps that are
unital.  A map $\phi\in \tmop{UCP}(S,B(K))$ is called \emph{pure} if
whenever $\phi-\psi$ is cp (we write $\phi\geq \psi$ in this case) for some $\psi\in\tmop{CP}(S,B(K))$, then there exists $0\leq t \leq 1$ such that
$\psi=t\phi$.  For example, when $K$ is one-dimensional, a pure ucp
map from $S$ to $B(K)$ is just a pure state.  When $K$ is finite-dimensional, the elements of $\tmop{UCP}(S,B(K))$ are called  \emph{matrix
  states}. Denote the set of pure matrix states from $S$ to $M_k$ by $\mathcal{P}_k(S)$, and let
$\mathcal{P}(S)=\bigcup_{i=1}^{\infty}\mathcal{P}_i(S)$. We show in
Theorem~\ref{purenorm} that for any $n$ and any $(s_{ij})\in M_n(S)$,
\begin{align*}
\|(s_{ij})\| &= \max_{\psi\in\mathcal{P}(M_n(S))}\|\psi((s_{ij}))\|.
\end{align*}
The above result, but with ``sup'' in place of ``max'', is contained
elsewhere: both in work of Farenick (\cite{Farenick04}), and in
unpublished work of Zarikian.

We will prove in Theorem~\ref{mainthm} that when $S$ is
separable, every pure matrix state on $S$ is a compression of a boundary
representation for $S$: if $\phi$ is in $\mathcal{P}(S)$, then there exist
$\pi\in \partial_S$ and an isometry $v$ such that
$\phi(\cdot)=v^{\ast}\pi(\cdot)v$.  This generalizes \cite[Theorem
8.2]{Arveson08}.  When we combine Theorem~\ref{mainthm},
Theorem~\ref{purenorm}, and a result of Hopenwasser (see
Remark~\ref{hopenwasser}), we obtain Theorem~\ref{mainresult}, the main result.

Let $X$ be a compact Hausdorff space, and let $M$ be a linear,
separating subspace of $C(X)$ that contains constants.  A
\emph{boundary} for $M$ is a subset $Y$ of $X$, not necessarily
closed, such that for any $f\in M$, there exists $y\in Y$ with $\|f\|=|f(y)|$.  In other words, a boundary for $M$ is a
norm-attaining subset of $X$. There is a rich theory of boundaries in
this setting, the highlight of which is a theorem of Bishop and de
Leeuw for uniform algebras (see \cite[Theorem 6.5]{BdL59} and \cite[p. 39]{Phelps01}).  A
natural extension of the definition of boundary to the case when $S$
is a concrete operator system in $A:=C^{\ast}(S)$ is afforded by equivalence classes of irreducible representations
of $A$. We denote this set by $\hat{A}$, and though it (the
\emph{spectrum} of $A$) is usually topologized, we consider it as merely
a set.  A \emph{boundary} for $S$ is a set $B\subseteq \hat{A}$ such
that for any $n$ and any $(s_{ij})\in M_n(S)$, there exists $[\pi]\in B$
with $\|(s_{ij})\|=\|(\pi(s_{ij}))\|$. Let $\ch{S}$ be the set of
unitary equivalence classes of boundary representations for $S$. We can 
translate the result indicated in the first paragraph into the
language of boundaries: when $S$ is separable, $\ch{S}$ is a
boundary for $S$.  This has immediate consequences for a
certain notion of peaking for operator systems introduced by Arveson
in \cite{Arveson08a}, which we discuss briefly in
Remark~\ref{finalrem}.  

It is possible to show that $\ch{S}$ is a boundary for $S$ by
considering pure states on $M_2(S)$.  The method below
is different and preferred for the light it sheds on pure ucp maps.  Also, many of the results that follow are
phrased in terms of concrete operator systems.  This is merely for
convenience.  The results can also be stated for unital operator
spaces, by noting the correspondence between unital completely
contractive maps on a unital operator space $V$ and ucp maps on the
operator system $V+V^{\ast}$ (see \cite[Proposition 1.2.8]{Arveson69} and \cite[Proposition 2.12]{Paulsen02}).

The collection of matrix states $(\ucp{S}{M_n})_{n\in \mathbb{N}}$ is closed under
finite direct sums and conjugation by isometries; this is the essential feature of a matrix convex set (defined in Section~\ref{four}). Webster and Winkler in \cite{Webster99} proved a Krein-Milman theorem
for compact matrix convex sets using matrix extreme points. In Theorem~\ref{bdrypt} we apply the results of Section~\ref{three}
on boundary representations to improve this result when $S$ is an operator system in a
matrix algebra, using a new notion of extremeness for matrix convex sets that corresponds exactly to boundary representations.

\section{Pure ucp maps}
 Given a linear map
$\phi:E\to F$ between vector spaces, define
$\phi^{(n)}:M_n(E)\to M_n(F)$ as
$\phi^{(n)}((x_{ij}))=(\phi(x_{ij}))$ for all $(x_{ij})\in
M_n(E)$.  We will use $1_{H}$ for the identity in $B(H)$ and $1_k$ for
the identity in $M_k$.  When the context is clear, we simply use 1 as
the identity for unital objects. If two operators $a$ and $b$ are unitarily equivalent,
we write $a\sim_u b$; we use the same notation for unitarily
equivalent representations.

There is an illuminating characterization of pure matrix states in
terms of certain extreme points of matrix convex sets (see
Section~\ref{four}). The full power of this characterization is not necessary here
--- we consider only an important special case. Let $x$ be in $B(H)$ and let $\os{x}$ be the operator system
$\tmop{span}\{x,x^{\ast},1\}$.  The set $\ucp{\os{x}}{M_n}$ encodes the
same information as the \emph{$n^{th}$-algebraic matricial range} of
$x$, which we denote by $W^n(x)$.  It is defined as
\begin{align*}
W^n(x) &:= \{\phi(x):\phi\in\ucp{\os{x}}{M_n}\}.
\end{align*}
Because any $\phi\in \ucp{\os{x}}{M_n}$ is determined by
$\phi(x)$, and any $a\in W^n(x)$ is the image of $x$ under some
$\psi\in \ucp{\os{x}}{M_n}$, we see that $W^n(x)$ and
$\ucp{\os{x}}{M_n}$ determine each other.

The algebraic matricial range is a generalization of the numerical range, and was introduced by Arveson in \cite{Arveson72}.  He observed
that it enjoys a particularly strong convexity property: it is closed under \emph{$C^{\ast}$-convex combinations}; that is, closed under sums of the form
\begin{align}\label{eq90}
\sum_{i=1}^m x_i^{\ast}a_ix_i,
\end{align}
where $a_i$ is in $W^n(x)$, $x_i$ is in $M_n$ for $i=1,2,\ldots,m$, and
$\sum_{i=1}^m x_i^{\ast}x_i=1_n$.  We call a subset of a
$C^{\ast}$-algebra 
\emph{$C^{\ast}$-convex} when it is closed under $C^{\ast}$-convex
combinations. Paulsen and Loebl (\cite{Paulsen81}) defined a
\emph{$C^{\ast}$-extreme} point of a $C^{\ast}$-convex set as an
element $a$ such that whenever $a$ is written as a
$C^{\ast}$-convex combination as in \eqref{eq90}, then under the additional
assumption that each $x_i$ is invertible, $a\sim_u a_i$ for $i=1,2,\ldots,m$.

It can be shown that $\phi\in \ucp{\os{x}}{M_n}$ is pure iff $\phi(x)$
is irreducible and $C^{\ast}$-extreme in $W^n(x)$.  This follows from
\cite[Theorem 5.1]{Farenick04} and an observation preceding Example
2.2 in \cite{Webster99}.

Morenz obtained a Krein-Milman theorem for a compact $C^{\ast}$-convex
set $\Gamma\subseteq M_n$.  He showed that $\Gamma$ is the
$C^{\ast}$-convex hull of certain $C^{\ast}$-extreme points, which are
themselves formed from what he called ``structural elements''.  Although we do not need to define this term, we explain below how structural elements appear when we apply Morenz's theorem to the compact $C^{\ast}$-convex set $W^n(y)$. 
\begin{theorem}[\cite{Morenz94}]\label{morenzthm}
Let $y$ be in $B(H)$ and let $n$ be in $\mathbb{N}$. The set $W^n(y)$ is the $C^{\ast}$-convex hull of its $C^{\ast}$-extreme points as follows: every $a\in W^n(y)$ is a $C^{\ast}$-convex combination of the form
\begin{align*}
a &= \sum_{i=1}^m x_i^{\ast}\psi_i(y)x_i,
\end{align*}
where each $\psi_i\in \ucp{\os{y}}{M_n}$ is such that either 
\begin{enumerateroman}
\item $\psi_i$ is in $\mathcal{P}_n(\os{y})$, or 
\item $\psi_i(y)\sim_u \alpha_i(y)\oplus t_i1_{n-l}$, for $\alpha_i\in
  \mathcal{P}_l(\os{y})$ for some $l<n$ and some $t_i \in \partial W^1(\alpha_i(y))$.
\end{enumerateroman}
We can also arrange that $m\leq 3n^2$.
\end{theorem}

When $y$ is in $M_n$, the $\psi_i(y)$'s or $\alpha_i(y)$'s (depending on whether we are in
case (i) or (ii) above) are the structural elements of $W^n(y)$.
Because $W^n(y)$ essentially is $\ucp{\os{y}}{M_n}$, Morenz's theorem
may be reinterpreted as a Krein-Milman theorem for
$\ucp{\os{y}}{M_n}$.  In fact, the structural elements of
$\ucp{\os{y}}{M_n}$ are exactly the boundary representations for
$\os{y}$ (\cite{Kleski12}); see also Remark~\ref{lastrem} for how this may
be applied to more general operator systems.

Before we make use of Morenz's theorem, we require a few preliminary results.

\begin{prop}\label{pureandextreme}
Let $S_1\subseteq S_2$ be operator systems with the same unit. If the ucp map $\phi:S_2\ra
B(H)$ is linearly extreme in $\tmop{UCP}(S_2,B(H))$ and $\phi|_{S_1}$ is
pure, then $\phi$ is pure.
\end{prop}
\begin{proof}
Write $\phi=\phi_1+\phi_2$ for $\phi_1,\phi_2\in
\tmop{CP}(S_2,B(H))$; we must show that $\phi_1$ and $\phi_2$ are scalar multiples of $\phi$. Of course, if we restrict $\phi$ and $\phi_1+\phi_2$ to $S_1$, we still have equality. Because $\phi|_{S_1}$ is pure, it follows that $\phi_1|_{S_1}=t\phi|_{S_1}$ for some $0\leq t\leq 1$; thus $\phi_1(1)=t\phi(1)=t1_H$. Similarly, we have  $\phi_2(1)=(1-t)\phi(1)=(1-t)1_H$.  Assuming that $0<t<1$, this implies
that $(1/t)\phi_1$ and $(1/(1-t))\phi_2$ are ucp. Now we can write
$\phi$ as a convex combination of ucp maps:
\begin{align*}
\phi &= t\cdot \frac{1}{t}\phi_1+(1-t)\cdot \frac{1}{1-t}\phi_2.
\end{align*}

Because $\phi$ is linearly extreme, we have
$\phi=(1/t)\phi_1=(1/(1-t))\phi_2$.  We conclude that $t\phi=\phi_1$ and
$(1-t)\phi=\phi_2$, which is what we wanted to show.
\end{proof}

We can endow the bounded operators from $S$ to $B(H)$ with a weak* topology, called the
\emph{bounded weak} or \emph{BW}-topology, via the identification of
this set with a dual Banach space.  In its relative BW-topology, $\ucp{S}{B(H)}$ is compact (see \cite[Section 1.1]{Arveson69} or \cite[Chapter 7]{Paulsen02} for more details).

\begin{cor}\label{pureextension}
Let $S_1\subseteq S_2$ be operator systems with the same unit.  Every pure ucp map on $S_1$ has
a pure extension to $S_2$.
\end{cor}

\begin{proof}
Let $\phi\in \ucp{S_1}{B(H)}$ be pure and let
\begin{align*}
\mathcal{F} &:= \{\psi\in \ucp{S_2}{B(H)} :\psi|_{S_1}=\phi\}.
\end{align*}
We claim that $\mathcal{F}$ is a face.  It is clearly convex and
BW-compact.  Also, if $t\psi_1+(1-t)\psi_2$ is in $\mathcal{F}$ for some
$0<t<1$ and $\psi_1,\psi_2\in \ucp{S_2}{B(H)}$, then
$t\psi_1|_{S_1}+(1-t)\psi_2|_{S_1}=\phi$. Because $\phi$ is pure, we must
have $\phi=\psi_1|_{S_1}=\psi_2|_{S_1}$, and this completes the claim.  Therefore $\mathcal{F}$ has an extreme point $\phi'$ which is an extreme
point of $\ucp{S_2}{B(H)}$. By Proposition~\ref{pureandextreme}, it
follows that $\phi'$ is pure.
\end{proof}

\begin{remark}
The above corollary is particularly useful when $S_1$ is a concrete
operator system and $S_2$ is $C^{\ast}(S_1)$.  In that case, any pure
ucp map $S_1\to B(H)$ has a pure ucp extension $C^{\ast}(S_1)\to
B(H)$. This was also noticed by Arveson (see the remarks following the
proof of \cite[Theorem 2.4.5]{Arveson69}), and proved by Farenick (but for pure matrix states --- see \cite[Theorem B]{Farenick00}).
\end{remark}

We now show the main result of this section.

\begin{theorem}\label{purenorm}
Let $S$ be a concrete operator system, not necessarily separable. For any  $s\in S$, there exists a pure matrix state $\phi$ on $S$ such that $\|\phi(s)\|=\|s\|$.
\end{theorem}

\begin{proof}
Let $s$ be in $S$. First, we show that we can find a matrix state on $S$
which realizes the norm of $s$. There exists a state $\gamma$ on $C^{\ast}(S)$ such
that $\gamma(s^{\ast}s)=\|s\|^2$.  By the GNS construction, there
exist a representation $\pi_{\gamma}$, a Hilbert space $H_{\gamma}$,
and a cyclic vector $\xi_{\gamma}\in H_{\gamma}$ such that
\begin{align*}
\|s\|^2\geq \|\pi_{\gamma}(s)\|^2\geq
\|\pi_{\gamma}(s)\xi_{\gamma}\|^2=\gamma(s^{\ast}s)=\|s\|^2.
\end{align*}
Let $v:\C{2}\to H_{\gamma}$ be an isometry whose image contains $\tmop{span}\{\xi_{\gamma},\pi_{\gamma}(s)\xi_{\gamma}\}$.  A routine calculation shows that $\|v^{\ast}\pi_{\gamma}(s)v\|=\|\pi_{\gamma}(s)\|$. Define
$\phi:\os{s}\to M_2$ by $\phi(a)=v^{\ast}\pi_{\gamma}(a)v$ for all $a\in
\os{s}$; we have $\|\phi(s)\|=\|s\|$.

Next, we show that we can find a \emph{pure} matrix state on $S$ realizing
the norm of $s$. By Theorem~\ref{morenzthm}, we may write $\phi(s)$ as a $C^{\ast}$-convex
combination of certain $C^{\ast}$-extreme points:
\begin{align*}
\phi(s) &= \sum_{i=1}^m x_i^{\ast}\phi_i(s)x_i, 
\end{align*}
where $\phi_i$ is in $\ucp{\os{s}}{M_2}$, $x_i$ is in $M_2$ for
$i=1,2,\ldots,m$ and $\sum_{i=1}^mx_i^{\ast}x_i=1_2$; and each $\phi_i(s)$ is $C^{\ast}$-extreme in
$W^2(s)$ in the way stated in the theorem.  Let
$x$ be the $2m\times 2$ matrix $\mat{x_1 & x_2 & \cdots &
  x_m}^T$.  Note that $\sum_{i=1}^mx_i^{\ast}x_i=1_2$ implies
$x^{\ast}x=1_2$, and so we can write the $C^{\ast}$-convex combination as a compression:
\begin{align}\label{eq50}
\phi(s) &= x^{\ast}\mat{\phi_1(s) & & \\ & \ddots &  \\ & & \phi_m(s)}x.
\end{align}
Recall that $\|s\|=\|\phi(s)\|$.  When we combine this with equation~\eqref{eq50} we obtain
\begin{align*}
\|s\| = \|\phi(s)\| =\left \| x^{\ast}\mat{\phi_1(s) & & \\ & \ddots &  \\ & & \phi_m(s)}x \right \| \leq \left \|\mat{\phi_1(s) & & \\ & \ddots &  \\ & & \phi_m(s)} \right \| = \max_i \|\phi_i(s)\|\leq \|s\|.
\end{align*}
Therefore  $\|s\|=\|\phi_j(s)\|$ for some $1\leq j\leq m$;
let $\psi=\phi_j$.   Now if $\psi$ is as in case (i) of
Theorem~\ref{morenzthm}, we apply Corollary~\ref{pureextension} to obtain a pure extension of $\psi$.  Otherwise, $\psi(s)$ is
unitarily equivalent to $\rho(s)\oplus \rho(s)$ for some pure state $\rho$
on $\os{s}$. By Corollary~\ref{pureextension}, $\rho$ has a pure extension to $S$.  In either case, we have a pure matrix state on $S$ which realizes the norm of $s$.
\end{proof}

\begin{remark}
The above theorem improves \cite[Theorem 2.2]{Pollack72} and
\cite[Theorem 4.7]{SmithWard80}. Both results show that for $n\geq 2$ and for any $s\in S$,
\begin{align*}
\|s\| &= \sup_{x\in W^n(s)}\|x\|.
\end{align*}
A stronger conclusion can be drawn. That this ``sup'' is a ``max'' is clear, because $\phi\mapsto
\phi(x)$ is a continuous map of $\ucp{S}{M_n}$ in its relative BW-topology to $W^n(x)$ in its relative weak topology --- which, in this
finite-dimensional setting, coincides with the norm topology.  Because
the former set is compact, so is $W^n(x)$, and we conclude that the
supremum is attained.  Yet despite this immediate stronger
conclusion, it is not clear that the norm is attained on a pure
matrix state; we have shown in Theorem~\ref{purenorm} that $C^{\ast}$-convexity theory yields a pure matrix state realizing the norm.

Farenick obtained a Krein-Milman theorem (\cite[Theorem 2.3]{Farenick04}) which implies that for any $n$
and any $(s_{ij})\in M_n(S)$,
\begin{align*}
\|(s_{ij})\| = \sup_{\psi\in\mathcal{P}(M_n(S))} \|\psi((s_{ij}))\|.
\end{align*}
This also follows from unpublished work of Zarikian. For any $n\in \mathbb{N}$, we may apply Theorem~\ref{purenorm} to the
operator system $M_n(S)$ to obtain the above result, improving ``sup'' to ``max''.
\end{remark}

\section{Boundary representations}\label{three}
We now state the main theorem.

\begin{theorem}\label{mainresult}
Let $S$ be a concrete separable operator system.  For each $s\in S$, there exists a boundary representation $\pi$ for $S$ such that $\|\pi(s)\|=\|s\|$.
\end{theorem}

To show this, we need some preliminary results. We first prove a lemma
modeled on \cite[Lemma 8.3]{Arveson08}.  Let $H$ be a 
Hilbert space and $(X,\mu)$ be a standard probability space;  suppose
$a_x$ is in $B(H)$ for all $x\in X$.  Assume that $x\mapsto \lambda(a_x)$
is a $\mathbb{C}$-valued Borel function for every vector functional
$\lambda$ on $B(H)$ (i.e. it is \emph{weakly measurable}). We use the expression 
\begin{align}\label{eq7}
b &= \int_X a_x\,d\mu(x)
\end{align} 
to mean that for any vector functional $\lambda$ on $B(H)$,
\begin{align}\label{eq8}
\lambda(b) &= \int_X \lambda(a_x)\,d\mu(x).
\end{align}
The operator $b$ is the \emph{weak integral} of the function $x\mapsto
a_x$, and in this case, equation~\eqref{eq8} in fact holds for every $\sigma$-weakly continuous functional $\lambda$.  Weak integrals can be generalized to Banach spaces; the interested reader may consult \cite{Diestel84} for more
information.  When we replace $B(H)$ with a locally convex vector
space $E$ and we suppose $X$ is a compact convex subset, then if equation~\eqref{eq8} holds for every $\lambda$ in a set of functionals on $E$ that
separates $X$, we say that $b$ is the \emph{barycenter} of the measure
$\mu$. 

Let $H$ be separable with orthonormal basis $\{e_i\}$. 
Equation~\eqref{eq7} says that the $(i,j)$ matrix entry of $b$ is
$\int_{X} \langle a_x e_{j}, e_{i} \rangle\,d\mu(x)$. We will be
interested in the case when an equation like \eqref{eq7} holds for
every $b$ in the image of a ucp map from a separable operator system $S$ into $B(H)$.  Define
\begin{align*}
\tmop{CP}_r(S,B(H)):=\{\psi\in
\tmop{CP}(S,B(H)):\|\psi\|\leq r\}.
\end{align*}
In \cite[Remark
4.2]{Arveson08}, it is shown that a map $X\ni x\mapsto \rho_x\in
\ucp{S}{B(H)}$ is
Borel measurable iff $x\mapsto \rho_x(a)$ is weakly measurable
for every self-adjoint $a\in S$.  This equivalence is also true if we replace $\ucp{S}{B(H)}$ by
$\tmop{CP}_r(S,B(H))$.  In other words,  $x \mapsto \rho_x(a)$ is
weakly measurable for all self-adjoint
$a\in S$ iff $x\mapsto \rho_x\in \tmop{CP}_r(S,B(H))$ is a Borel map.

The following lemma says that if $\rho_x$ is in
$\tmop{CP}_r(S,B(H))$ for all $x\in X$ and $\phi$ is in a face of
$\tmop{CP}_r(S,B(H))$, and $\phi(a)$ is the weak
integral of $\rho_x(a)$ for all $a\in S$, then almost every $\rho_x$
is in the face.

\begin{lemma}\label{arvlemma}
Let $S$ be a separable operator system, $H$ be a separable Hilbert space,
$\phi$ be in a face $\mathcal{F}$ of $\tmop{CP}_r(S,B(H))$, and $(X,\mu)$ be a 
standard 
probability space.  Suppose $\rho_x\in
\tmop{CP}_r(S,B(H))$ for each $x\in X$ and $x\mapsto \rho_x(a)$ is
weakly measurable for each $a\in S$, and that  
\begin{align}\label{eq107}
\phi(a) = \int_X \rho_x(a)\,d\mu(x)
\end{align}
for all $a\in S$.  Then
for a.e. $x$, $\rho_x$ is in $\mathcal{F}$.
\end{lemma}
\begin{proof}
The set $\tmop{CP}_r(S,B(H))$ is BW-compact,
convex, and because $S$ and $H$ are separable, it is also metrizable. Define a Borel measure
$\nu$ on $\tmop{CP}_r(S,B(H))$ by $\nu(E)=\mu\{x:\rho_x\in
E\}$ for any Borel set $E\subset \tmop{CP}_r(S,B(H))$. Let $\psi$ be
the barycenter of the measure $\nu$; we claim that $\psi=\phi$. Let
$\Lambda:=\{L_{\gamma,s}:\gamma\in B(H)_{\ast}, s\in
S\}$, where $L_{\gamma,s}(\psi):=\gamma\circ\psi(s)$ for all $\psi\in
\tmop{CP}_r(S,B(H))$. Any $L\in \Lambda$ is a Borel map from
$\tmop{CP}_r(S,B(H))$ to $\mathbb{C}$ (indeed, $\Lambda \subseteq \tmop{CP}_r(S,B(H))^{\ast})$, so we have
\begin{align}\label{eq4}
\int_{X} L(\rho_x)\,d\mu(x) &= \int_{\tmop{CP}_r(S,B(H))} L(\rho)\,d\nu(\rho).
\end{align}
We can now write  
\begin{align*}
L(\phi) = \int_{\tmop{CP}_r(S,B(H))} L(\rho)\,d\nu(\rho) = L(\psi),
\end{align*}
where the first equality follows equations~\eqref{eq4} and \eqref{eq107},
and the second equality follows from the fact that $\psi$ is the
barycenter of $\nu$. The set $\Lambda$ is separating for
$\tmop{CP}_r(S,B(H))$, since an element $\sigma$ of the latter set is
determined by $S$, and each $\sigma(s)$ is determined by its matrix
entries. This establishes that $\psi=\phi$.  By Bauer's theorem
(\cite[Theorem 9.3]{Simon11}), $\nu(\tmop{CP}_r(S,B(H))\setminus
\mathcal{F})=0$. Thus for a.e. $x$, $\rho_x$ is in $\mathcal{F}$.
\end{proof}

Let $\phi\in \ucp{S}{B(H)}$.  A
\emph{dilation} of 
$\phi$ is a ucp map $\phi':S\to B(K)$, $K\supseteq H$, such that
$\phi'(a)=p_{H}\psi(a)|_{H}$ for all $a\in S$ (where $p_H$ is the
projection of $K$ onto $H$).  A ucp map $\phi$ is called
\emph{maximal} if whenever $\psi$ dilates $\phi$, then $\psi=\phi\oplus \rho$, for some ucp map
$\rho$.  Muhly and Solel showed the significance of maximal ucp maps for the theory of boundary
representations (though in the language of Hilbert modules) in
\cite{MS98}, where they proved that for a representation $\pi$ of $C^{\ast}(S)$, $\pi|_{S}$ has the UEP iff
$\pi|_{S}$ is maximal.  Thus $\pi$ is a boundary representation
for $S$ iff $\pi|_{S}$ is pure and maximal.  In \cite{DM05}, Dritschel and McCullough showed that every ucp
map on $S$ actually has a maximal dilation, building on earlier work of
Agler (\cite{Agler88}).  This allowed them to conclude that every operator system has
sufficiently many representations whose restrictions to $S$ are
maximal.  Arveson then showed in \cite{Arveson08} that when $S$ is separable, those
representations can be taken to be boundary representations using
disintegration theory.  We will use these ideas in the next theorem, where we show that every pure matrix state is a compression of a
boundary representation.

\begin{theorem}\label{mainthm}
Let $S$ be a concrete separable operator system and let
$\phi$ be in $\mathcal{P}_n(S)$. Then $\phi$ has an extension to $C^{\ast}(S)$
of the form $y^{\ast}\pi(\cdot) y$, where $\pi$ is a boundary
representation for $S$ and $y$ is an isometry.
\end{theorem}

\begin{proof}
The following diagram captures the setup of the proof.
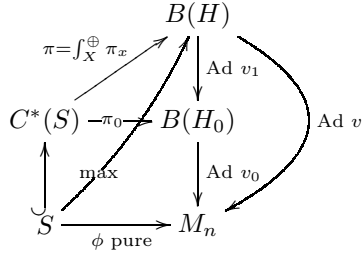
\[
\xymatrix{
                  & B(H) \ar[d]^{\tmop{Ad}\,v_1} \ar@/^3.5pc/[dd]^{\tmop{Ad}\,v} \\
C^{\ast}(S)  \ar[r]|<<<<{\pi_0}  \ar[ru]^{\pi=\int_{X}^{\oplus}\pi_x} &  B(H_0)  \ar[d]^{\tmop{Ad}\,v_0}       \\
S  \ar@/_.5pc/[ruu]|<<<<<<{\text{max}}  \ar[r]_{\phi\text{ pure}} \ar@{^{(}->}[u]^{}               & M_n
} 
\]
We explain below.

The pure ucp map $\phi$ has an extension to $C^{\ast}(S)$, and Stinespring's theorem allows us to write $\phi(\cdot)=v_0^{\ast}\pi_0(\cdot) v_0$ for a representation
$\pi_0$ acting on a separable Hilbert space $H_0$ and an isometry $v_0$. By the main result of
\cite{DM05} explained above, we can find maximal dilation of
$\pi_0|_{S}$ acting on a separable Hilbert space $H\supseteq H_0$; $p_{H_0}(\cdot)|_{H_0}$ is implemented by an isometry $v_1$.  
Now extend that maximal dilation to $C^{\ast}(S)$ as a representation $\pi$.  We have $\phi(a)=v^{\ast}\pi(a) v$ for all
$a\in S$ (where $v=v_1v_0$).  From the proof of \cite[Theorem
7.1]{Arveson08}, there is a standard probability space $(X,\mu)$ such
that $\pi$ has a disintegration $\int_X^{\oplus} \pi_x\,d\mu(x)$
with respect to the Hilbert space $H=\int_X^{\oplus}H_x\,d\mu(x)$ and for
a.e. $x$,  $\pi_x$ is a boundary representation. Thus
\begin{align*}
\phi(a) =  v^{\ast}\pi(a)v = v^{\ast}\left (\int_{X}^{\oplus}\pi_x(a)\,d\mu(x) \right) v
\end{align*}
for all $a\in S$.  We want to rewrite this expression as a weak
integral, in order to use Lemma~\ref{arvlemma}.

Let $\{e_i:i=1,2,\ldots,n\}$ be the standard 
basis for $\C{n}$, and let $\xi_i\in H$ be $ve_i$ for each $i=1,2,\ldots,n$. Then
\begin{align*}
v^{\ast}\pi(a)v &= \mat{\langle\pi(a)\xi_j,\xi_i\rangle}
\end{align*} 
for all $a\in S$.  By the disintegration
$H=\int_{X}^{\oplus}H_x\,d\mu(x)$, the vectors $\xi_i$ have the form
$(\xi_i(x))$, where $\xi_i(x)\in H_x$ for all $x\in X$ and $x\mapsto \xi_i(x)$ is
square-integrable for $i=1,2,\ldots,n$. Thus we can rewrite this last expression:
\begin{align*}
v^{\ast}\pi(a)v &= \mat{\int_X\langle\pi_x(a)\xi_j(x),\xi_i(x)\rangle\,d\mu(x)}.
\end{align*}
Define $v_x:\C{n}\to H_x$ by $v_xe_i=\xi_i(x)$ for each $i=1,2,\ldots,n$
and $x\in X$.  Each $v_x$ is not necessarily contractive; nevertheless, we have 
\begin{align*}
v_x^{\ast}\pi_x(a)v_x &= \mat{\langle\pi_x(a)\xi_j(x),\xi_i(x)\rangle}.
\end{align*}

To show that the map $x\mapsto v_x^{\ast}\pi_x(a)v_x$ is weakly measurable
for all $a\in S$, fix $a\in S$ and choose $z,w$ in $\C{n}$. We
compute:
\begin{align*}
\langle v_x^{\ast}\pi_x(a)v_xz,w\rangle &= \langle \pi_x(a)v_xz,v_x
w\rangle \\
&=  \sum_{i,j=1}^n \langle z,e_i\rangle \langle e_j,w\rangle \langle \pi_x(a)
\xi_i(x) ,\xi_j(x) \rangle.
\end{align*}
The function $x\mapsto \pi_x(a)$ is weakly measurable (see \cite[IV.8]{Takesaki1}). So the above
is a finite sum of measurable functions, and thus is measurable. We can now write
\begin{align}\label{eq11}
\phi(a) &= \int_{X}v_x^{\ast}\pi_x(a)v_x\,d\mu(x)
\end{align}
for all $a\in S$.

We are tempted to use Lemma~\ref{arvlemma} on the above expression,
except that even in this finite-dimensional setting, it is not clear 
that $\tmop{Ad}\,v_x\circ \pi_x\in \tmop{CP}_r(S,B(H))$, no matter what
value is chosen for $r$. To get around this difficulty, we normalize
the measure $\mu$.  Define $d\nu(x) := \|v_x\|^2\,d\mu(x)$; we will apply this to
the set $X':=\{x\in X:v_x\neq 0\}$. Let $t:=\nu(X')$ and let $y_x:=\|v_x\|^{-1}v_x$ for $x\in X'$.  Then assuming that $0<t<\infty$, the probability measure $t^{-1}\nu$ yields an equation similar to equation~\eqref{eq11}:
\begin{align}\label{eq1}
\int_{X'} y_x^{\ast}\pi_x(a)y_xt^{-1}\,d\nu(x) &=  t^{-1}\int_{X'}(\|v_x\|^{-1}v_x)^{\ast}\pi_x(a)(\|v_x\|^{-1}v_x)\|v_x\|^2\,d\mu(x)\notag \\
&= t^{-1}\int_{X}v_x^{\ast}\pi_x(a)v_x\,d\mu(x) \\
&= t^{-1}\phi(a)\notag
\end{align} 
for all $a\in S$.  In order to apply Lemma~\ref{arvlemma} to equations~\eqref{eq1}, we must
show that 
\begin{enumerateroman}
\item  the map $X'\ni x\mapsto y_x^{\ast}\pi_x(a)y_x$ is
weakly measurable for all $a\in S$;
\item $0<t<\infty$ and so $t^{-1}\nu$ is a probability measure on $X'$; 
\item there exists $0<r<\infty$ such that $t^{-1}\phi\in
  \tmop{CP}_r(S,B(H))$ and $\tmop{Ad}\,y_x\circ \pi_x\in \tmop{CP}_r(S,B(H))$ for each $x\in X'$;
\item $t^{-1}\phi$ lies in a face of $\tmop{CP}_r(S,B(H))$.
\end{enumerateroman}
We have shown that $X\ni x\mapsto v_x^{\ast}\pi_x(a)v_x$
is weakly measurable for all $a\in S$, so to prove (i), it suffices to prove
that $X'\ni x\mapsto \|v_x\|^{-2}$ is measurable.  Let $\{z_i\}$ be a
norm-dense subset of the unit ball of $\C{n}$; since $X\ni x\mapsto
v_x^{\ast}\pi_x(a)v_x$ is weakly measurable when $a=1$, we see that
$x\mapsto \|v_xz_n\|^{2}$ is measurable for each $n$. If we take the
supremum over $n$, the resulting function $X\ni x\mapsto \|v_x\|^{2}$ is
seen to be measurable (see also \cite[A 77]{Dixmier77}).  This implies
$X'\ni x\mapsto \|v_x\|^{-2}$ is measurable.  For item (ii), we can show that $1\leq
t\leq n$ by showing that $1\leq \nu(X)\leq n$, since $t=\nu(X')=\nu(X)$:
\begin{align*}
 1=\int_{X} \|v_xe_1\|^2\,d\mu(x)\leq
 \int_{X}\|v_x\|^2\,d\mu(x)=t\leq
 \int_{X}\tmop{tr}\,v_x^{\ast}v_x\,d\mu(x)= n.
\end{align*} 
To prove (iii), note that
$t^{-1}\leq 1$, and that $\|\tmop{Ad}\,y_x\circ\pi_x\|=\|y_x^{\ast}y_x\|= 1$, which shows
that $t^{-1}\phi$ and  $\tmop{Ad}\,y_x\circ \pi_x$ are in $\tmop{CP}_1(S,B(H))$ for $x\in
X'$.  Lastly, for item (iv),  let $\mathcal{F}:=\{l\phi: 0\leq l\leq
1\}$. The set $\mathcal{F}$ is a face of $\tmop{CP}_1(S,B(H))$ because $\phi$ is
pure, and clearly $t^{-1}\phi$ is in $\mathcal{F}$.

 We can now apply Lemma~\ref{arvlemma} to conclude that
$y_x^{\ast}\pi_x(\cdot)y_x$ is in $\mathcal{F}$ for a.e. $x\in X'$, so there exists $l_x\in [0,1]$ such that $l_x\phi(\cdot)=y_x^{\ast}\pi_x(\cdot)y_x$ for a.e. $x\in X'$. It follows that $l_x1_n=y_x^{\ast}y_x\neq 0$ for a.e. $x\in X'$. Thus the operator $l_x^{-1/2}y_x$ is an isometry (and so $v_x$ is a multiple of an isometry) for
a.e. $x\in X'$.  Finally, $\phi(\cdot)=(l_x^{-1/2}y_x)^{\ast}\pi_x(\cdot)(l_x^{-1/2}y_x)$
for a.e. $x\in X'$, so $\phi$ is a compression of the
boundary representation $\pi_x$ for a.e. $x\in X'$.
\end{proof}

Theorem~\ref{mainthm} generalizes Arveson's result on pure states on
$S$ (\cite[Theorem 8.2]{Arveson08}) to pure ucp maps on $S$: he showed that every pure state on $S$ can be extended to a state $\gamma$ on $C^{\ast}(S)$ whose GNS representation $\pi_{\gamma}$ is a boundary representation for $S$. There are obstacles to successfully adapting the above method to the
infinite-dimensional setting, most of which depend on whether
or not every $v_x$ (or a.e. $v_x$) is bounded.  Even if the appropriate measurability
conditions are satisfied --- so equation~\eqref{eq11} is valid ---
the measure $\nu$ may not be finite.  Without this crucial fact, we
cannot obtain equations~\eqref{eq1}, and so we cannot appeal to
Lemma~\ref{arvlemma}.

With these preliminary results, the proof of the main result follows easily.

\begin{proof}[Proof of Theorem~\ref{mainresult}]
By Theorem~\ref{purenorm}, there is a pure matrix state $\phi:S\ra M_k$,
for some $1\leq k\leq 2$, such that $\|\phi(s)\|=\|s\|$. By
Theorem~\ref{mainthm}, we can find a boundary representation $\pi$ for
$S$ and
an isometry $v$ such that $\phi(a)=v^{\ast}\pi(a) v$ for all $a\in S$. Then
\begin{align*}
\|s\| = \|\phi(s)\| = \|v^{\ast}\pi(s) v\| \leq \|\pi(s)\| \leq \|s\|
\end{align*}
and so $\|s\|=\|\pi(s)\|$.   
\end{proof}

\begin{remark}\label{hopenwasser}
To realize the norm of $(s_{ij})\in M_n(S)$ on $\ch{S}$, we can apply the above results to the operator system $M_n(S)$ as follows: by Theorem~\ref{mainresult}, there is a boundary representation $\pi:C^{\ast}(M_n(S))\to B(H)$ for $M_n(S)$ such that $\|\pi((s_{ij}))\|=\|(s_{ij})\|$. This representation is unitarily equivalent to $\sigma^{(n)}$, where $\sigma$ is an irreducible representation of $C^{\ast}(S)$. By the main result of \cite{Hopenwasser73}, this $\sigma$ is a boundary representation for $S$. Thus $[\sigma]\in \ch{S}$ realizes the norm of $(s_{ij})$.
\end{remark}

\begin{remark}\label{finalrem}
In \cite{Arveson08a}, Arveson defined a \emph{peaking representation}
for a concrete operator system $S$ to be an irreducible representation
$\pi$ of $C^{\ast}(S)$ such that there exist an $n$ and an $(s_{ij})\in M_n(S)$ satisfying
\begin{align*}
\|\pi^{(n)}((s_{ij}))\| > \|\sigma^{(n)}((s_{ij}))\|
\end{align*}
for all irreducible representations $\sigma\nsim_u \pi$.  It follows
immediately from Theorem~\ref{mainresult} that when $S$ is separable, all
peaking representations are boundary representations.

Let $X$ be a compact metrizable space, and suppose $M$ is a linear, uniformly closed, separating subspace of $C(X)$ that contains constants.  The set of peak points for $M$ is dense in the Choquet boundary for $M$; when $M$ is a uniform algebra, the set of peak points for $M$ is exactly the Choquet boundary for $M$ (\cite{BdL59}). This is in stark contrast to the noncommutative case: there are operator algebras with no peaking representations.  For example,
let $x$ be the unilateral shift on $B(\ell^2)$, and let $\oa{x}$ be the
operator algebra generated by $x$. Recall that the spectrum of
$C^{\ast}(x)$ can be identified with $\{\tmop{id}\}\cup \mathbb{T}$
(see \cite[Example VII.3.3]{Davidson96}). The quotient map
$C^{\ast}(x)\to C^{\ast}(x)/\mc{K}\cong C(\mathbb{T})$ is completely
isometric on $\oa{x}$; the irreducible representations
parametrized by $\mathbb{T}$ are exactly the boundary representations
for $\oa{x}$. The boundary representations are quotients of the
identity representation, so none can be peaking for $\oa{x}$.  The
identity also cannot be peaking for $\oa{x}$, since for any
$(s_{ij})\in M_n(\oa{x})$, there exists a boundary representation for $\oa{x}$ realizing the norm by Theorem~\ref{mainresult}.  We conclude that $\oa{x}$ has no peaking representations.  

\end{remark}

We intend to explore in a later paper the conditions under which an operator system has peaking representations, and when analogues of classical results for peaking phenomena hold in the noncommutative setting.

\section{Applications to operator systems in matrix algebras}\label{four}
Let $E$ be a locally convex vector space.  A \emph{matrix convex set} in $E$ is a collection $\bs{K}=(K_n)_{n\in\mathbb{N}}$ of sets $K_n\subseteq M_n(E)$ such that every sum of the form
\begin{align}\label{eq60}
a &= \sum_{i=1}^m v_i^{\ast} a_i v_i
\end{align}
is in $K_n$, where $a_i$ is in $K_{n_i}$ and $v_i$ is in $M_{n_i,n}$ for $i=1,2,\ldots,m$, and
$\sum_{i=1}^m v_i^{\ast}v_i=1_n$.  Equivalently, a matrix
convex set in $E$ is a collection $\bs{K}$ of sets $K_n\subseteq M_n(E)$ that
is closed under finite direct sums and compressions.  From now on, we will
abbreviate $(K_n)_{n\in \mathbb{N}}$ as $(K_n)$. The
definition of matrix convex set is due to Wittstock
(\cite{Wittstock84}); some important properties of matrix convex sets were
proved in \cite{Effros97}.  In \cite{Webster99}, Webster and Winkler showed a number of interesting results on matrix convex sets.   Below we outline some of their work, which, when combined with the results from the previous sections, yields a connection to boundary representations.

We will be interested in the case when each $K_n$ is compact in the
product topology on $M_n(E)$, and we refer to such $\bs{K}$ as 
\emph{compact} matrix convex sets.  The matrix convex combination
\eqref{eq60} is \emph{proper} when each $v_i$ is surjective.  An
element $a$ is called \emph{matrix extreme} if whenever it is written
as a proper matrix convex combination as in \eqref{eq60}, then
$a\sim_u a_i$ for each $i=1,2,\ldots,m$ (\cite[Definition
2.1]{Webster99}). Let $\bs{\partial K}$ denote the set of matrix
extreme points of $\bs{K}$ and let $\overline{\tmop{co}}(\bs{\partial K})$ be the
closed matrix convex hull of $\bs{\partial K}$. Webster and Winkler
proved that $\overline{\tmop{co}}(\bs{\partial K})=\bs{K}$ when
$\bs{K}$ is compact (\cite[Theorem 4.3]{Webster99}).  They also showed that every compact matrix convex set ``is'' the collection of matrix state spaces of an operator system as follows: a \emph{matrix affine} mapping on $\bs{K}$ is a collection $\bs{\theta}:=(\theta_n)$ of maps $\theta_n:K_n\to M_n(F)$ for a vector space $F$ such that 
\begin{align*}
\theta_n\left [ \sum_{i=1}^m v_i^{\ast}a_iv_i\right ] &= \sum_{i=1}^m v_i^{\ast}\theta_{n_i}(a_i)v_i
\end{align*}
where $\sum_{i=1}^m v_i^{\ast}a_iv_i$ is a matrix convex combination in $K_n$. If each $\theta_n$ is a homeomorphism, then $\bs{\theta}$
is a \emph{matrix affine homeomorphism}.  Let $A(\bs{K})$ denote the
set of matrix affine mappings from $\bs{K}$ to $\mathbb{C}$. 
Remarkably, this is an (abstract) operator system, and $\bs{K}$ and 
$(\ucp{A(\bs{K})}{M_n})$ --- which is a compact matrix convex set in
$A(\bs{K})^{\ast}$ --- are matrix affinely homeomorphic
(\cite[Proposition 3.5]{Webster99}). For example, a compact matrix
convex set in $\mathbb{C}$ is $(W^n(x))$ for some Hilbert space
operator $x$. (This fits nicely with the observation in
\cite[Proposition 31]{Paulsen81} that $W^n(x)$ is the prototypical
compact $C^{\ast}$-convex set in $M_n$.)  We will exploit the
identification of $\bs{K}$ and $(\ucp{A(\bs{K})}{M_n})$ repeatedly in
what follows.  We adopt the following notation:
\begin{align}\label{eq102}
\bs{K} &\longleftrightarrow (\ucp{A(\bs{K})}{M_n}) \notag\\
a &\longmapsto \phi_a \\
a_{\psi} &\longmapsfrom \psi.\notag
\end{align}
 Using this identification, Farenick showed (\cite[Theorem
B]{Farenick00}) that the matrix extreme points of $\bs{K}$ are exactly the
pure ucp maps in $(\ucp{A(\bs{K})}{M_n})$.

A ucp map $\phi:S_1 \to S_2$ between operator systems is a
\emph{complete order isomorphism} if $\phi$ has an inverse which is
also ucp.  In this case, $S_1$ and $S_2$ are isomorphic as operator
systems.  By a fundamental result of Choi and Effros (\cite{ChoiEffros77}), an abstract
operator system $S$ can be realized
as a concrete operator system: there exist a Hilbert space $H$ and a
complete order injection $\phi:S\to B(H)$ (i.e. $S$ is completely
order isomorphic to its image in $B(H)$).  From now on, we will assume
without loss of generality that $A(\bs{K})$ is concrete.

There are several ways to characterize boundary representations for
operator systems in matrix algebras (see \cite{Blecher07},
\cite[4.3.7]{BlecherLeMerdy}, \cite{Arveson08c}); here we present another that
shows a connection between boundary representations for $A(\bs{K})$ and a certain type of extreme point of
$\bs{K}$.  
\begin{definition}
A \emph{boundary point} of a matrix convex set $\bs{K}$ is
an element $b\in K_n$ such that whenever $b$ is a matrix convex
combination 
\begin{align}\label{eq106}
b &= \sum_{i=1}^m v_i^{\ast}a_iv_i,
\end{align}
not necessarily proper, of elements $a_i\in K_{n_i}$, then $a_i\sim_u
b$ if $n_i\leq n$; otherwise, $a_i\sim_u b\oplus c_i$ for some
$c_i\in \bs{K}$.
\end{definition}
The motivation for this definition is the following: a matrix extreme
point $b\in K_n$  is an element that cannot be written as a matrix convex
combination of elements that appear ``below'' it in the
hierarchy $\ldots,K_{n+1},K_n,K_{n-1},\ldots,K_1$.   One would
like to define a notion of extremeness that also rules out being a
matrix convex combination of elements ``above'' in the hierarchy ---
except in a trivial way --- and the definition of boundary point does this.   Evidently, every boundary point is a matrix
extreme point, but not
every matrix extreme point is a boundary point.  For example, let $x\in M_3$ be 
\begin{align*}
   x  &=\begin{array}{lll}x_1\oplus x_2, & x_1=1, & x_2=\mat{0 & 2 \\ 0 & 0},
     \end{array}
\end{align*}
and let $\bs{K}$ be the matrix convex set $(W^n(x))$. It is easy to
see that $x_1\in W^1(x)$ and $x_2\in W^2(x)$ are matrix extreme points
of $\bs{K}$, but because $x_1$ is a proper compression of the irreducible matrix $x_2$, it
cannot be a boundary point.  Nevertheless, when $A(\bs{K})$ acts on a finite-dimensional Hilbert space, $\bs{K}$ has ``enough'' boundary points. 

\begin{theorem}\label{bdrypt}
Let $\bs{K}$ be a compact matrix convex set in a locally convex vector
space $E$.  Suppose $A(\bs{K})$ acts on a finite-dimensional
Hilbert space.  The boundary points of $\bs{K}$ correspond exactly to the boundary representations for $A(\bs{K})$.
\end{theorem}
\begin{proof}
We may assume that $A(\bs{K})\subseteq M_l$ for some $l$.  The
collection $(\ucp{A(\bs{K})}{M_n})$ is a compact matrix convex set in
$A(\bs{K})^{\ast}\subseteq M_l^{\ast}$. Applying the Webster-Winkler theorem in this finite-dimensional
setting, we have
\begin{align*}
\tmop{co}(\bs{\partial} (\ucp{A(\bs{K})}{M_n})) &= (\ucp{A(\bs{K})}{M_n}).
\end{align*}
Let $b\in K_n$ be a boundary point of $\bs{K}$; identify it with its
image $\phi_b$ in $\ucp{A(\bs{K})}{M_n}$.  We show that $\phi_b$ is
unitarily equivalent to the restriction of a boundary representation
to $A(\bs{K})$.  Write $\phi_b$ as a proper matrix convex combination of
matrix extreme points $\{\phi_1,\phi_2,\ldots,\phi_m\}\subset (\ucp{A(\bs{K})}{M_n})$; each $\phi_i$
is pure by \cite[Theorem B]{Farenick00}.  It follows from the
definition of boundary point that $\phi_b\sim_u \phi_i$ for $i=1,2,\ldots,m$, so $\phi_b$ is pure.  Theorem~\ref{mainthm} implies that the pure matrix state $\phi_b$ is a compression of a boundary representation $\pi$ for $A(\bs{K})$: 
\begin{align}\label{eq100}
\phi_b(\cdot)=v^{\ast}\pi(\cdot) v,
\end{align}
for some isometry $v$. The representation $\pi$ acts on a finite-dimensional Hilbert space
(since $C^{\ast}(A(\bs{K}))\subseteq M_l$), so $\pi|_{A(\bs{K})}\in
(\ucp{A(\bs{K})}{M_n})$. Thus \eqref{eq100} is a matrix convex
combination in $(\ucp{A(\bs{K})}{M_n})$. The ucp map $\pi|_{A(\bs{K})}$ is
pure (\cite[Lemma 2.4.3]{Arveson69}).  If $v$ is a proper isometry,
then by the definition of boundary point, $\phi_b$ is a direct summand
of a unitary conjugate of $\pi|_{A(\bs{K})}$, which contradicts the fact that
$\pi|_{A(\bs{K})}$ is pure.  Thus $v$ is unitary.

Now suppose that $\pi$ is a boundary representation for $A(\bs{K})$ acting on
$\C{n}$; identify $\pi|_{A(\bs{K})}$ with its image $b_{\pi}\in K_n$.  Suppose
$b_{\pi}$ is a matrix convex combination
\begin{align*}
b_{\pi} &= \sum_{i=1}^m v_i^{\ast}a_iv_i,
\end{align*}   
where $a_i$ is in $K_{n_i}$ for $i=1,2,\ldots,m$.  We show that if $n_i\leq
n$, then $b_{\pi}\sim_u a_i$; otherwise, there exists $c_i\in \bs{K}$ such that $a_i\sim_u b_{\pi}\oplus
c_i$.  Using \eqref{eq102}, we may rewrite the above equation as
\begin{align}\label{eq104}
\pi|_{A(\bs{K})}(\cdot) &= \sum_{i=1}^m v_i^{\ast}\phi_{a_i}(\cdot)v_i.
\end{align}
It follows that $\pi|_{A(\bs{K})}\geq
\tmop{Ad}\,v_i\circ \phi_{a_i}$ for $i=1,2,\ldots,m$. Fix $j\in
\{1,2,\ldots,m\}$. The ucp map $\pi|_{A(\bs{K})}$ is pure, so there exists $t_j\in [0,1]$ such that $t_j\pi|_{A(\bs{K})}(\cdot)=v_j^{\ast}\phi_{a_j}(\cdot)v_j$. This
implies $t_j1_r=v_j^{\ast}v_j$.  Assuming that $t_j\neq
0$, it follows that $t_j^{-1/2}v_j$ is an isometry and
$\pi|_{A(\bs{K})}(\cdot)
=(t^{-1/2}v_j)^{\ast}\phi_{a_j}(\cdot)(t^{-1/2}v_j)$.  Therefore, if
$n_j\leq n$, we conclude that in fact $n_j=n$.  This forces
$t_j^{-1/2}v_j$ to be unitary.  Otherwise, $t_j^{-1/2}v_j$ is a
proper isometry.  Because $\pi|_{A(\bs{K})}$ is maximal, we
must have $\phi_{a_j}\sim_u \pi|_{A(\bs{K})}\oplus \psi$ for some $\psi$ in
$(\ucp{A(\bs{K})}{M_n})$, which implies $a_j\sim_u b_{\pi}\oplus a_{\psi}$.
\end{proof}

Farenick identified matrix extreme points of
$\bs{K}$ with pure ucp maps in $(\ucp{A(\bs{K})}{M_n})$. Now assume
that $A(\bs{K})$ acts on a finite-dimensional Hilbert space. In the above
theorem, we identified boundary points of $\bs{K}$ with boundary
representation for $A(\bs{K})$. We know from Theorem~\ref{mainthm} that every pure matrix state of $A(\bs{K})$ is a compression of a boundary representation for $A(\bs{K})$.  Using \eqref{eq102}, we get as a corollary that every matrix extreme
point of $\bs{K}$ is a compression of a boundary point of $\bs{K}$.
We can apply this to get another simple corollary: the set of boundary points of $\bs{K}$ is the minimal subset of $\bs{K}$ that recovers $\bs{K}$.

\begin{cor}
Let $\bs{K}$ be a compact matrix convex set in a locally convex vector
space $E$.  Suppose $A(\bs{K})$ acts on a finite-dimensional
Hilbert space. Let $\Gamma$ be a subset of $\bs{K}$.  Then $\tmop{co}(\Gamma)=\bs{K}$ iff for every boundary point $b\in
\bs{K}$, there are an isometry $v$ and an element $g$ of $\Gamma$ such that
$b=v^{\ast}gv$. 
\end{cor} 
\begin{proof}
Assume $\tmop{co}(\Gamma)=\bs{K}$. Let $b\in K_n$ be a boundary point
of $\bs{K}$. By assumption, we may write it as a matrix convex
combination of $g_1,g_2,\ldots,g_m$ in $\Gamma$:
\begin{align*}
b &= \sum_{i=1}^m v_i^{\ast}g_iv_i.
\end{align*}
Identify $\bs{K}$ with $(\ucp{A(\bs{K})}{M_n})$ as in \eqref{eq102}, so $K_n\ni b\mapsto
\phi_b\in \ucp{S}{M_n}$.  By Theorem~\ref{bdrypt}, $\phi_b$ is pure
and maximal.  We may use the same techniques as those following equation~\eqref{eq104} to conclude that $\phi_b$ is a compression of $\phi_{g_i}$ (assuming that
$v_i^{\ast}g_iv_i\neq 0$) for each $i=1,2,\ldots,m$.  We conclude $b$
is a compression of $g_i$ for $i=1,2,\ldots,m$.

Now suppose that for every boundary point $b\in
\bs{K}$, there are an isometry $v$ and $g\in \Gamma$ such that
$b=v^{\ast}gv$.  Let $a$ be in $\bs{K}$; we want to show that $a$ is in
$\tmop{co}(\Gamma)$.  Use the Webster-Winkler theorem to write $a$ as a
matrix convex combination of matrix extreme points. By the result
mentioned above, each matrix extreme point is a compression of a
boundary point.  Thus we may write $a$ as a (not-necessarily proper) matrix convex combination:
\begin{align*}
a &= \sum_{i=1}^m v_i^{\ast}b_iv_i,
\end{align*} 
where $b_i$ is a boundary point for $i=1,2,\ldots,m$. By assumption,
there exist $g_i\in \Gamma$ and an isometry $y_i$ such that
$b_i=y_i^{\ast}g_iy_i$ for $i=1,2,\ldots,m$. Thus
\begin{align*}
a &= \sum_{i=1}^m (y_iv_i)^{\ast}g_i(y_iv_i),
\end{align*} 
which is a matrix convex combination of elements of $\Gamma$.
\end{proof}

\begin{remark}\label{lastrem}
Let $E$ be a locally convex vector space. There is an obvious way to
define $C^{\ast}$-convexity in $M_l(E)$;
consequently, when $\Gamma$ is  compact and $C^{\ast}$-convex in $M_l(E)$, we may
apply Morenz's definition of structural element (\cite[Definition 2.1 and Definition
2.3]{Morenz94}) to $\Gamma$.  Now
suppose $S$ is an operator system acting on $\C{l}$.  The set
$\ucp{S}{M_l}$ is a compact $C^{\ast}$-convex subset of $M_l(S^{\ast})$, and one can show that the structural elements of
this set are exactly the boundary points of $(\ucp{S}{M_n})$.  This,
and Theorem~\ref{bdrypt}, show that any two of the following three sets are in 1-1 correspondence: the boundary points of $(\ucp{S}{M_n})$, the boundary
representations for $S$, and the structural elements of $\ucp{S}{M_l}$ (\cite{Kleski12}).
\end{remark}

\begin{acknowledgements}
I thank my Ph.D. advisor, David Sherman, who was enormously helpful during the
preparation of this paper.
\end{acknowledgements}

\bibliographystyle{amsalpha}
\bibliography{docs}

\end{document}